\documentclass[11pt,a4paper]{article} 
\usepackage{epsf,epsfig,amsfonts,amsgen,amsmath,amstext,amsbsy,amsopn,amsthm}
\usepackage{amsmath}
\usepackage{amsfonts,amsthm,amssymb,bm}
\usepackage{amsfonts}
\usepackage{graphics}
\usepackage{latexsym,bm}
\usepackage{amsfonts,amsthm,amssymb,bbding}
\usepackage{indentfirst}
\usepackage{graphicx}
\usepackage{color}
\usepackage[colorlinks=true,anchorcolor=blue,filecolor=blue,linkcolor=red,urlcolor=blue,citecolor=blue]{hyperref}     
\usepackage{float}
\usepackage{tikz,enumerate}
\usepackage{geometry} 
\newtheorem{thm}{Theorem}[section]
\newtheorem{conj}[thm]{Conjecture}
\newtheorem{lemma}[thm]{Lemma}
\newtheorem{claim}{Claim}[section]

\newtheorem{defn}[thm]{Definition}
\newtheorem{cor}[thm]{Corollary}

\newtheorem{prop}[thm]{Proposition}

\newtheorem{example}[thm]{Example}

\renewcommand{\leq}{\leqslant}
\renewcommand{\le}{\leqslant}
\renewcommand{\geq}{\geqslant}
\renewcommand{\ge}{\geqslant}

\begin{document}
\title{An exponentially small gap of the Perron vector\\
on independent sets}

\author{
Hongzhang Chen\thanks{School of Mathematics and Statistics, 
Gansu Center for Applied Mathematics, 
Lanzhou University, Lanzhou, Gansu, 730000, China. Email: \url{mnhzchern@gmail.com}.} 
\and 
Jianxi Li\thanks{School of Mathematics and Statistics, Minnan Normal University, Zhangzhou, Fujian, 363000, China. 
Email: \url{ptjxli@hotmail.com}. Partially supported by the NSF of Fujian Province (No. 2021J02048).}
\and 
Yongtao Li\thanks{Corresponding author. Yau Mathematical Sciences Center (YMSC), Tsinghua University, Beijing, 100084, China. 
Email: \url{ytli0921@hnu.edu.cn}.} 
\and 
Lele Liu\footnote{School of Mathematical Sciences, Anhui University, Hefei, 230601,
 China. Email: \url{liu@ahu.edu.cn}. 
 Supported by the NSFC (No. 12471320), and Anhui Provincial NSF for Excellent Young Scholars (No. 2408085Y003).} 
\and 
Bo Ning\footnote{College of Cryptology and Cyber Science \& College of Computer Science, Nankai University, Tianjin, 300350, China.
Email: \url{bo.ning@nankai.edu.cn}. 
Partially supported by the NSFC (No. 12371350) and Fundamental Research Funds for the Central Universities, Nankai University (No. 63243151).}
}

\date{\today}
\maketitle

\begin{abstract}
A classical result of Cioab\u{a} states that if $G$ is a connected graph with the unit Perron vector $\bm{x}$, then any independent set $S$ of $G$ satisfies $\sum_{v\in S} x_v^2 \le \frac{1}{2}$, with equality if and only if $G$ is a bipartite graph and $S$ is one of the partite sets. 
Let $\chi(G)= k $ be the chromatic number of $G$. A well-known conjecture of Gregory asserts that any independent set $S$ of $G$ satisfies $\frac{1}{2} - \sum_{v\in S}x_v^2 = \Omega ((k/n)^{1/2})$. 
Recently, Liu and Ning [{\it J. Combin. Theory Ser. B 176 (2026)}] disproved Gregory's conjecture by constructing a graph $G$ and an independent set $S$ such that $\frac{1}{2}- \sum_{v\in S}x_v^2 = O(k^5/n^3)$. 
Furthermore, they conjectured that this bound is tight up to a constant factor. 
In this paper, we first show that any cycle $C_n$ with odd integer $n\ge 7$ provides a simple counterexample to Gregory's conjecture.
Second, we establish that for any independent set $S$, we have $\frac{1}{2} - \sum_{v\in S}x_v^2 =  \frac{q}{4\lambda -2q}$, where $\lambda$ is the spectral radius of $G$, and $q$ is the Rayleigh quotient of $\bm{x}$ restricted to $\bar{S} :=V(G)\setminus S$. 
Third, we construct a graph with arbitrarily large chromatic number and find an independent set $S$ such that $\sum_{v\in S}x_v^2$ can be arbitrarily close to $\frac{1}{2}$, with an exponentially small gap. 
Our construction shows that there is no universal lower bound of the form $\Omega (k^{\alpha}/n^{\beta})$ for any $\alpha, \beta >0$. 
This settles both Gregory's original conjecture and the modified conjecture of Liu and Ning in the negative. 
Finally, we show the tightness of our construction and provide some local weighted lower bounds. 
\end{abstract}



\section{Introduction}
Let $G=(V,E)$ be a simple connected graph of order $n$. The {\it adjacency matrix} of $G$ is 
defined as $A(G)=[a_{ij}]_{i,j=1}^n$, where $a_{ij}=1$ if $ij\in E(G)$, and $a_{ij}=0$ otherwise. 
The {\it spectral radius} $\lambda (G)$ is defined as the maximum modulus of eigenvalues of $A(G)$. 
The Perron--Frobenius theorem implies that
the spectral radius $\lambda (G)$ is the largest eigenvalue of $A(G)$.  
For a connected graph $G$,  there exists a  positive unit eigenvector $\bm{x} = (x_v)_{v\in V}$ satisfying $\sum_{v\in
V} x_v^2 = 1$ with $x_v > 0$ for all $v\in V$. This vector is called
the {\it Perron vector} (or principal eigenvector) of $G$.  

Spectral graph theory is an important area as it reveals deep connections between the combinatorial structure of a graph and the algebraic properties of its associated matrices. 
In recent years, spectral graph theory has also found very important applications in extremal combinatorics \cite{Hua2019, JTYZZ2021}. 
By studying eigenvalues and eigenvectors, it provides powerful tools for analyzing graph properties that are otherwise hard to compute. In particular, spectral methods yield bounds on fundamental parameters such as the independence number \cite{GN2008, Hae2021} and the chromatic number \cite{Wil1967,Bil2006, WE2013}. 
Let $\alpha (G)$ be the independence number, and let $n_{\ge 0}(G)$ be the minimum number 
of non-negative eigenvalues taken over all Hermitian weighted adjacency matrices of $G$. 
The well-known inertia bound  states that $\alpha (G)\le n_{\ge 0}(G)$. 
Recently, Kwan and Wigderson \cite{KW2024} proved that this bound is far from tight. They obtained the largest possible gap between $\alpha(G)$ and $n_{\ge 0}(G)$, showing that there exists an $n$-vertex regular graph $G$ with $\alpha (G)=O(n^{3/4})$ and $n_{\ge 0}(G)=\Omega (n)$. 
Furthermore, Tang, Zhang and Elphick \cite{TZE2025} 
constructed an $n$-vertex graph $G$ such that 
$\alpha (G)\le 2$ and $n_{\ge 0}(G) =\Omega (n^{1/9})$. 
Consequently, we see that $n_{\ge 0}(G)$ cannot be upper bounded by any function of $\alpha (G)$.

\smallskip 
In this paper, we investigate the distribution of the Perron vector on independent sets for graphs with given chromatic number. 
For a subset $S\subseteq V(G)$, 
the {\it Perron mass} of $S$ is denoted by
\[ \sigma (S):= \sum_{v\in S} x_v^2. \]  
The {\it independent Perron mass} of $G$ is defined as 
\[ \sigma_{\mathrm{ind}} (G) :=\max_{S} \sum_{v\in S} x_v^2,\]
where the maximum is taken over all independent sets $S \subseteq V(G)$. 
This parameter is closely related to the local properties of $G$, and is of considerable interest. 
It is well-known (see, e.g., \cite{BFP2008}) that if $G$ is a bipartite graph with $V(G)=A\sqcup B$, then $\sigma (A) = \sigma (B) = \frac{1}{2}$. This identity is useful in the proof of a spectral stability result for a Tur\'{a}n-type problem \cite[Lemma 2.5]{LLZ-spectral-ESS}. 
The study of the coordinates of the Perron vector has always been an important problem that has attracted great attention in spectral graph theory; see \cite{LN2026,CH2026} for recent progress.

\smallskip
A well-known result in spectral graph
theory states that a connected graph 
is bipartite if and only if the negative of the largest eigenvalue is also an eigenvalue of the adjacency matrix.  
In 2010, Cioab\u{a} \cite{Cio2010} proved that for any connected graph $G$, we have
$\sigma_{\mathrm{ind}}(G) \le 1/2$, with equality if and only if $G$ is bipartite. 
To some extent, 
this result provides an interesting  characterization involving $\sigma_{\mathrm{ind}}(G)$ for a connected graph $G$ to be
bipartite. Moreover, 
this result naturally raises an intriguing 
 problem of bounding the quantitative gap $\frac{1}{2} - \sigma_{\mathrm{ind}}(G)$
for a graph $G$ with higher chromatic number  $\chi(G)=k\ge 3$. David A. Gregory proposed the following conjecture,
which was collected and introduced by Cioab\u{a}~\cite{Cio2021} at an online conference on spectral graph theory.

\begin{conj}[Gregory, cf. \cite{Cio2021}]  \label{conj-sigma}
    Let $G$ be a connected graph on $n$ vertices with chromatic number $\chi (G)=k\ge 3$. Then 
    $\sigma_{\mathrm{ind}} (G) \le \sigma_{\mathrm{ind}} (S_{k-1,n})$, 
    where $S_{k-1,n} := K_{k-1}\vee I_{n-k+1}$ is the split graph obtained by joining a clique $K_{k-1}$ to an independent set of size $n-k+1$.  
\end{conj}

In other words, the split graph $S_{k-1,n}$ achieves the maximum value of $\sigma_{\mathrm{ind}} (G)$ among all $n$-vertex graphs $G$ with chromatic number $k$. 
Incidentally, we mention here that split graphs often serve as the extremal graphs of several spectral graph problems in the literature; we refer the readers to \cite{Tait2019,CDT2024,LZS2024,LZZ2025, LLZ-color-critical}. 
By direct computation, we can check that (see, e.g., \cite{LN2026})
\[  \sigma_{\mathrm{ind}} (S_{k-1,n}) = \frac{1}{2} - \frac{k-2}{2\sqrt{(k-2)^2 +4(k-1)(n-k+1)}} . \]
Equivalently, Conjecture \ref{conj-sigma} can be rewritten as the following numerical form.  

\begin{conj}[Gregory, cf. \cite{Cio2021}]\label{conj-Gre}
Let $G$ be a connected graph on $n$ vertices with $\chi(G) =k $, and let $\bm{x}=(x_v)_{v\in V}$ be the unit Perron vector of $G$. 
Suppose that $S$ is an independent set of $G$.  Then
\begin{equation}\label{eq-Gre}
  \frac{1}{2} - \sum_{v\in S} x_v^2 \ge \frac{k-2}{2\sqrt{(k-2)^2 + 4(k-1)(n-k+1)}}.
\end{equation}
\end{conj}

This conjecture was also selected by Liu and Ning  \cite[Conjecture 23]{LN2023} as one of ``{\it Unsolved problems in spectral graph theory}''. 
In the case $k=2$, the right-hand side in (\ref{eq-Gre}) is $0$, which is consistent with  
Cioab\u{a}'s theorem. In the case $k\ge 3$, the bound (\ref{eq-Gre}) implies $\frac{1}{2} - \sum_{v\in S} x_v^2 >0$. The aim of Conjecture \ref{conj-Gre} is to provide a
nontrivial universal lower bound. 

\smallskip 
In 2026, Liu and Ning~\cite{LN2026} made the first progress on Conjecture \ref{conj-Gre}.
They surprisingly disproved Conjecture~\ref{conj-Gre} by constructing a graph $G$ 
and exhibiting an independent set $S$ in $G$ such that
$ \frac12 - \sum_{v\in S} x_v^2 = O(k^5/n^3)$, which is much smaller than the
conjectured bound of (\ref{eq-Gre}). 
Inspired by this result, they further proposed a modified conjecture as follows.

\begin{conj}[Liu--Ning \cite{LN2026}]\label{conj-LN}
Let $G$ be a connected graph of order $n$ with chromatic number $\chi(G)=k\ge 3$. 
Suppose that $S$ is an independent set of $G$. Then 
\[  \frac{1}{2} - \sum_{v\in S} x_v^2 \ge \Omega\left( \frac{k^5}{n^3} \right). \]  
\end{conj}

\subsection{An exponentially small gap}

In this paper, we construct a simple graph with arbitrarily large chromatic number $k$ and an independent set $S$ for 
which the gap between $\frac12 $ and $ \sum_{v\in S}x_v^2$ 
is exponentially small. 
Our construction not only refutes the original conjecture of Gregory, but also disproves the modified
polynomial-gap conjecture of Liu and Ning.

\begin{thm} \label{thm-main}
For every integer $k\ge3$ and all sufficiently large $n$, there exist a connected
graph $G$ on $n$ vertices with $\chi(G) = k$ and an independent set $S\subseteq V(G)$
such that
\[
\frac12 - \sum_{v\in S} x_v^2 
\le e^{ -\Omega_k(n\log n)}.
\]
\end{thm}

The exponential decay established in Theorem \ref{thm-main} is far smaller than any polynomial bound. Thus, 
Theorem \ref{thm-main} implies that no universal lower bound of the form $\Omega(k^\alpha / n^\beta)$
 can hold for any reals $\alpha,\beta>0$. We state it as the following corollary. 
Consequently, no lower bound depending polynomially on $n$ and $k$ can hold in this setting.

\begin{cor}
    There do not exist real numbers $C>0$ and $\alpha,\beta >0$ such that every connected graph $G$ of order $n$ with $\chi (G)=k$ satisfies $\frac{1}{2} - \sum_{v\in S} x_v^2 \ge Ck^{\alpha}/ n^{\beta} $ for every independent set $S$.  
\end{cor}

In fact, we shall prove the following more general result than Theorem \ref{thm-main}. 

\begin{thm} \label{thm-main-2}
    Let $k\ge 3$, $\ell \ge 2$ and $m\ge 2k$. 
    There exist a connected graph $G=G(m,\ell ,k)$ with order $n:=2m +2\ell +k$ and 
    $\chi (G)=k$, 
    and an independent set $S$ with  
    $$\frac{1}{2} - \sum_{v\in S} x_v^2 \le \frac{30k}{(m-1)^{4\ell +3}}.$$ 
\end{thm}

\subsection{Description of the construction in Theorem \ref{thm-main-2}}

We define the graph $G = G(m,\ell,k)$ in the following simple way.

\begin{defn} \label{defn-construction}
Take a complete bipartite graph $K_{m,m}$ with bipartition
$(S_A, Y_A)$, where $|S_A| = |Y_A| = m$. 
Take a path $P$ of order $2\ell$ with vertices 
$p_1, p_2, \dots, p_{2\ell}$.
Connect $p_1$ to a fixed vertex $y^* \in Y_A$.
Take a clique $K_k$ with vertex set $C = \{c_1, c_2, \dots, c_k\}$. Connect $p_{2\ell}$ to $c_1$; see Figure \ref{fig-G}. 
\end{defn}

 \begin{figure}[htbp]
\centering
\includegraphics[scale=1]{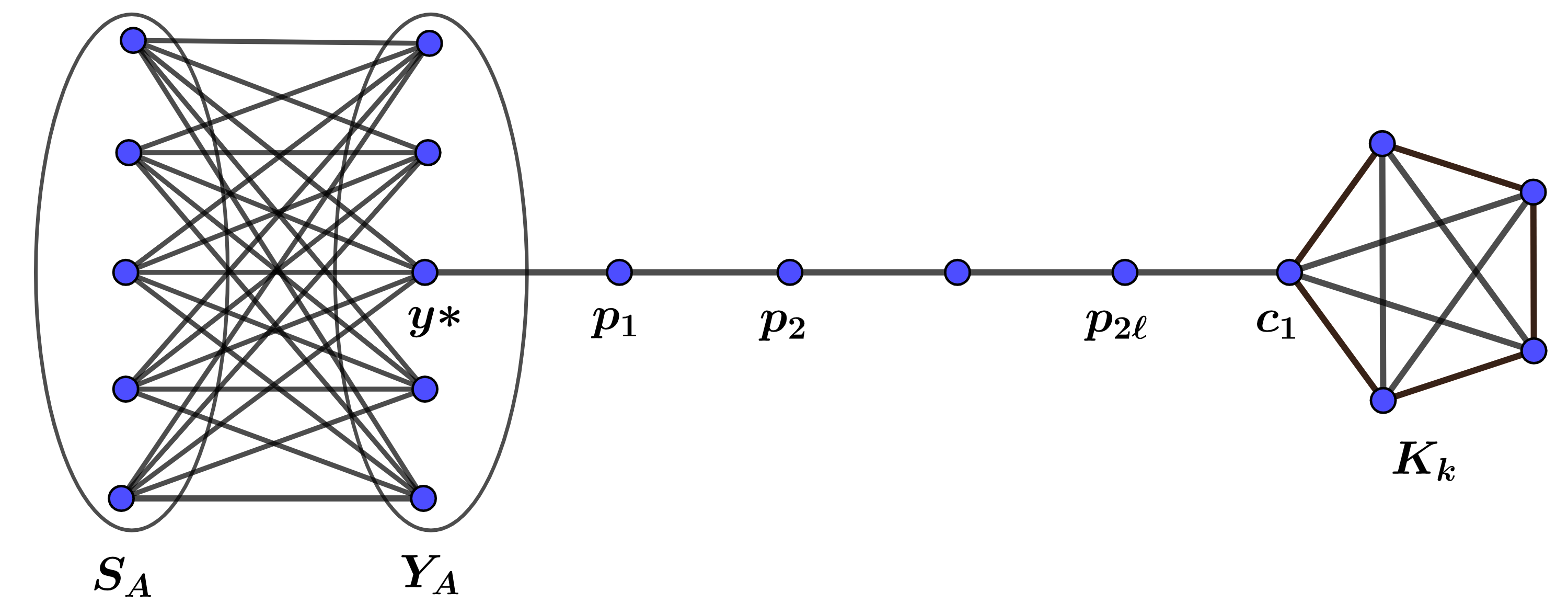} 
\caption{The graph $G(m,\ell ,k)$}
\label{fig-G}
\end{figure}

The graph $G(m,\ell,k)$ has $n= 2m + 2\ell + k$ vertices. 
Since it contains a clique $K_k$, we have $\chi (G)\ge k$. 
Conversely, the path $P_{2\ell}$ and the bipartite subgraph $K_{m,m}$ can be colored using colors already available outside the clique, so 
we get $\chi (G)=k$. Moreover, the set 
$S_{\text{path}} := \{p_1, p_3, \dots, p_{2\ell-1}\}$ 
contains no adjacent pair and has no edges to $S_{A}$, hence 
\begin{equation} \label{eq-defn-S}
    S := S_A \cup S_{\text{path}}
\end{equation}  
is independent. 
We will prove that $S$ is the desired  independent set with exponentially small gap. 

\smallskip 
\paragraph{Motivation.} 
Before presenting the proof argument, 
we would like to give an explanation of the original ideas behind such a construction. Motivated by Cioab\u{a}'s bound  $\sigma_{\mathrm{ind}}(G)\le {1}/{2}$, 
with equality if and only if $G$ is bipartite, one expects from a stability perspective that if $\sigma_{\mathrm{ind}}(G)$ is close to ${1}/{2}$, then $G$ should be close to a large bipartite graph after removing a small number of edges. So we embed a large complete bipartite subgraph $K_{m,m}$ into $G$.   
To guarantee $\chi (G)=k$, we simply embed a clique $K_k$ into $G$. To reduce the mutual influence between the coordinates of the Perron vector $\bm{x}$ on the bipartite graph $K_{m,m}$ and the coordinates on the clique $K_k$, we should connect a sufficiently long path between them. Moreover, adding such a path is a natural and direct way to ensure that $G$ is connected. 

\paragraph{Our approach.} 
The key insight of our construction is that the
exponential decay of the Perron vector along a long path
makes the coordinate of $p_{2\ell}$ and all coordinates of vertices in the clique $K_k$ sufficiently small. 
This forces the gap $\frac12-\sigma(S)$ to be exponentially tiny, where we write $\sigma (S)=\sum_{v\in S} x_v^2$ for short.  
We shall show that the coordinates of the vertices of $S_A$ and $Y_A$ are of order $m^{-1/2}$. In contrast, the $i$-th path vertex $p_i$ is of order $m^{-i-1/2}$; in particular, 
the last path vertex $p_{2\ell}$ is of order $m^{-2\ell-1/2}$, which yields that the vertices of the clique are of order $m^{-2\ell-3/2}$ and thereby become negligible when the length $\ell$ is sufficiently large; see Lemma \ref{lem-2}.  
Consequently, we can see that 
$\sum_{v\in S_A} x_v^2 = \frac{1}{2} - O(\frac{1}{m})$ and $\sum_{v\in Y_A} x_v^2 = \frac{1}{2} - O(\frac{1}{m})$. 
Moreover, we prove that the total contribution from the path and clique is negligible, i.e.,  $R:=\sum_{i=1}^{2\ell} x_{p_i}^2 + \sum_{j=1}^k x_{c_j}^2 = O_k(\frac{1}{m^2})$. These estimates imply a slightly weak bound $\frac{1}{2}- \sigma (S) < \frac{6}{m}$; see Proposition \ref{prop-weak}. 
Therefore, the main difficulty in our argument lies in deriving an exponentially small gap between $\frac{1}{2}$ and $\sigma (S)$. To overcome this, we need to capture the essential relation to the structure of the induced subgraph $G[\bar{S}]$ and establish an identity for the gap $\frac{1}{2} -\sigma (S)$; see Theorems \ref{thm-identity} and \ref{thm:sum-uv}. 

\paragraph{Organization.} 
The rest of this paper is organized as follows. In Section \ref{sec:simple}, we provide a simple counterexample for Conjecture \ref{conj-Gre}. In Section \ref{sec:identity}, we establish two useful identities for $\frac{1}{2} -\sigma (S)$. In Section \ref{sec:exponent}, we give the proofs of Theorems \ref{thm-main} and \ref{thm-main-2}. 
In Section \ref{sec:tight}, we demonstrate the tightness of the bound in Theorem \ref{thm-main} by showing $\frac{1}{2}- \sigma (S) \ge e^{-O(n\log n)}$; see Theorem \ref{thm-lower}.

\section{A simple counterexample for Gregory's conjecture}

\label{sec:simple}

Although Liu and Ning \cite{LN2026} already provided a general construction disproving
Conjecture~\ref{conj-Gre}, 
we next show that 
odd cycles provide a simple and transparent
counterexample.

\begin{example}\label{example-1}
Let $n\ge 7$ be an odd integer and $C_n$ be the cycle on $n$ vertices.
Choose $S$ to be any maximum independent set of $C_n$. Then $\chi(C_n)=3$, and
Conjecture~\ref{conj-Gre} is false.
\end{example}

\begin{proof}
 Since $C_n$ is $2$-regular, we know that  
$\lambda (C_n) = 2$ and the unit Perron eigenvector satisfies  
$x_v = \frac{1}{\sqrt{n}}$ for all $v\in V(C_n)$. 
Let $S$ be a maximum independent set with $|S|=\frac{n-1}{2}$. Then
$\sum_{v\in S} x_v^2 = \frac{n-1}{2} \cdot \frac{1}{n} = \frac{n-1}{2n}$. 
The right-hand side of (\ref{eq-Gre}) with $k=3$ becomes
$\frac12 - \frac{1}{2\sqrt{8n-15}}$. 
It is easy to verify that 
$\frac{n-1}{2n} > \frac12 - \frac{1}{2\sqrt{8n-15}}$ since 
 $n\geq 7$. Thus, Conjecture~\ref{conj-Gre} fails.
\end{proof}

\smallskip
\noindent 
{\bf Remark.} 
The odd cycle counterexample captures the essential reason: the conjectured gap in (\ref{eq-Gre}) is of order $\Theta(n^{-1/2})$, whereas for $C_n$, the actual gap is of order $O(n^{-1})$.

\section{Two exact identities for the gap}

\label{sec:identity}

For a connected graph $G$, let $S$ be an independent set
of $G$. We denote its complement by  $\bar{S}=V\setminus S$. 
We write $B$ for the 
$(S, \bar{S})$-incidence matrix, where $b_{vu}=1$ if $v\in S$, $u\in\bar{S}$ and
$vu\in E(G)$; and $b_{vu}=0$ otherwise. 
Let $A_{\bar{S}}$ denote the adjacency matrix of the induced subgraph $G[\bar{S}]$. 
Since $S$ is an independent set, the adjacency matrix $A:=A(G)$
can be partitioned as
\[
A = \begin{pmatrix} 0 & B \\ B^{\mathsf{T}} & A_{\bar{S}} \end{pmatrix}.
\]
Correspondingly, 
we partition the Perron 
vector $\bm{x}$  as $\bm{x}=\binom{\bm{y}}{\bm{z}}$, where $\bm{y}$ and $\bm{z}$ correspond to the vertices of $S$ and $\bar{S}$, respectively. 
Then $A \bm{x} = \lambda \bm{x}$ gives 
\begin{align} \label{e-3}
\lambda \bm{y} = B \bm{z} \quad \text{and} \quad 
\lambda \bm{z} = B^{\mathsf{T}} \bm{y} + A_{\bar{S}}\bm{ z}.
\end{align}  
In what follows, we establish an identity for 
$\frac{1}{2} - \sigma(S)$ that
involves the spectral radius of $G$ and the Rayleigh quotient of $\bm{z}$ restricted to the complement $\bar{S}$.

\begin{thm}\label{thm-identity}
For a connected graph $G=(V,E)$, let $S\subseteq V$ be an
independent set of $G$, and set $\bar{S}=V\setminus S$. With the notation above, let
$q := {\bm{z}^{\mathsf{T}} A_{\bar{S}} \bm{z}}/{\|\bm{z}\|^2}$.
Then 
\begin{equation*} 
\frac{1}{2} - \sigma (S) =  \frac{q}{4\lambda - 2q}.
\end{equation*}
\end{thm}

\begin{proof}
Note that $\|\bm{y}\|^2 =\sigma(S)$ and $\|\bm{ z}\|^2 = 1 - \sigma(S)$. 
From (\ref{e-3}), we have  
\begin{equation*}
 \bm{y}^{\mathsf{T}} B \bm{z} = \bm{y}^{\mathsf{T}} (\lambda \bm{y}) 
 = \lambda \| \bm{y}\|^2 = \lambda\, \sigma(S).
\end{equation*}
By the definition of $q$, we have
\begin{equation*}
\bm{z}^{\mathsf{T}} A_{\bar{S}} \bm{z} = 
q \|\bm{z}\|^2 = q(1-\sigma(S)).
\end{equation*} 
Since $\bm{
x}$ is a unit Perron vector of $G$,
and $S$ is an independent set, we obtain  
\[
\lambda = \bm{x}^{\mathsf{T}} A \bm{x} = 2\bm{y}^{\mathsf{T}}B\bm{z} + \bm{z}^{\mathsf{T}} A_{\bar{S}} \bm{z} = 2\lambda  \,\sigma(S) + q (1-\sigma(S)). 
\] 
Rearranging gives 
$\lambda - q = 2\lambda \,\sigma(S) - q\, \sigma(S) = (2\lambda -
q)\sigma(S)$. 
Since $A_{\bar{S}}$ is a principal submatrix of $A(G)$, the Rayleigh quotient satisfies 
$q\le \lambda (G[\bar{S}])\le \lambda $. Hence, we get $2\lambda - q \ge \lambda> 0$. Then 
$\sigma(S) = \frac{\lambda - q}{2\lambda - q} = 
\frac{1}{2} - \frac{q}{4\lambda -2q}$.
This completes the proof.
\end{proof}

Since $q$ is nonnegative, Theorem \ref{thm-identity} 
implies $\sigma (S)\le \frac{1}{2}$ for any independent set $S$, 
recovering  Cioab\u{a}'s bound.  
In addition, applying Theorem \ref{thm-identity}, we obtain the following corollary. 

\begin{cor} 
For a connected graph $G$, let $S\subseteq V$ be an
independent set of $G$, and set $\bar{S}=V\setminus S$. Let
$\mu $ be the smallest
eigenvalue of the induced subgraph $G[\bar{S}]$. Then
\[
\frac{1}{2} - \sigma(S) \ge 
\frac{\mu}{4\lambda -2\mu}, \]
with equality if and only if $\bm{z}$
is an eigenvector of $A_{\bar{S}}$ corresponding to $\mu$.
\end{cor}

\begin{proof}
Denote $f(q) :=  \frac{q}{4\lambda - 2q}$
for each $q \le \lambda$. 
We can see that $f(q)$ is strictly increasing on $(-\infty, \lambda)$.
The Rayleigh quotient
$q ={\bm{z}^{\mathsf{T}} A_{\bar{S}} \bm{z}}/{\|\bm{z}\|^2}$
satisfies
$ \mu \le q \le \lambda$, 
which yields 
\[ \frac{1}{2} - \sigma(S) 
\ge \frac{\mu}{4\lambda - 2\mu}.
\]
Equality occurs when $q = \mu$ and 
$\bm{z}$ is an eigenvector of $A_{\bar{S}}$
corresponding to the eigenvalue $\mu$.
\end{proof}

Next, we present a simple equivalent identity.

\begin{thm}\label{thm:sum-uv}
Let $G$ be a connected graph with the spectral radius $\lambda$ and unit Perron vector $\bm{x}=(x_v)_{v\in V(G)}$. Let $S\subseteq V(G)$ be an independent set and denote   $\bar{S}:=V(G)\setminus S$. Then
\[
\frac12 - \sigma (S)= \frac1\lambda\sum_{uv\in E(\bar{S})}x_u x_v.
\]
\end{thm}

\begin{proof}
Since $S$ is an independent set, by the eigenvalue equation, we have
\[ \lambda \, \sigma (S) = 
\lambda \sum_{u\in S} x_u^2 = \sum_{u\in S} x_u \cdot \lambda \, x_u
= \sum_{u\in S} x_u\sum_{v\in N(u)} x_v
= \sum_{uv\in E(S,\bar{S})} x_u x_v.
\]
Also, we see that 
\[
\lambda = \bm{x}^{\mathsf{T}} A(G) \bm{x}
= 2\sum_{uv\in E(S,\bar{S})} x_u x_v + 2\sum_{uv\in E(\bar{S})} x_u x_v.
\]
Substituting the first identity into the second gives 
\[
\lambda = 2\lambda \, \sigma (S) + 2\sum_{uv\in E(\bar{S})} x_u x_v.
\]
By rearranging, we get the desired identity  immediately. 
\end{proof}

\section{Proofs of Theorems \ref{thm-main} and \ref{thm-main-2}}

\label{sec:exponent}

In this section, we give the proofs of Theorems \ref{thm-main} and \ref{thm-main-2}. 
To begin with, we present the upper bounds on all coordinates of the Perron vector of the graph  $G(m,\ell ,k)$.

\subsection{Spectral analysis of the Perron vector}

Let $G$ be a graph with minimum degree $\delta (G)$ and maximum degree $\Delta (G)$. It is well-known that $\delta (G) \le \lambda (G)\le \Delta (G)$, where equality holds if and only if $G$ is a regular graph. 
For $G=G(m,\ell,k)$ and $\lambda=\lambda(G)$, the graph properly contains $K_{m,m}$ as a subgraph, so $\lambda>m$. Also $\Delta(G)=m+1$ and $G$ is not regular, so we get $\lambda <\Delta(G)=m+1$. Thus $m <\lambda < m+1$. 

We now analyze the coordinates of the Perron vector of
$G$. 
The following lemma roughly says that the Perron vector of $G$ decays exponentially along the path $P$ starting from the bipartite subgraph $K_{m,m}$ to the clique $K_k$. The coordinates of the path vertices decrease by a factor related to $\lambda$ at each step, the clique vertices are strictly smaller than the last path vertex $p_{2\ell}$, and most of the mass of $\bm{x}$ is concentrated on the bipartite subgraph $K_{m,m}$. 

\begin{lemma}\label{lem-2}
The unit Perron vector $\bm x = (x_v)_{v \in V(G)}$ of $G$
satisfies the following properties:
\begin{enumerate}[(i)]
\item Denote $p_0 := y^* \in Y_A$. 
For every $i = 1, 2, \dots, 2\ell$,
\[
x_{p_i} < \frac{x_{p_{i-1}}}{\lambda - 1}.
\]

\item Every vertex $c_i$ in the clique $C$ satisfies that for all $ j = 1,2,\dots,k$, 
\[
x_{c_j} < \frac{x_{p_{2\ell}}}{\lambda - k + 1}.
\]

\item By symmetry,
all vertices in $S_A$ share the same eigenvector coordinate,
denoted by $\alpha$;
all vertices in $Y_A\setminus\{y^*\}$ share the
same coordinate, denoted by $\beta$.
Write $\beta^* := x_{y^*}$. Then 
\[
\tfrac{1}{3}m^{-1/2} < \alpha < m^{-1/2},\qquad 
\tfrac{1}{4}m^{-1/2} < \beta < m^{-1/2},\qquad 
\tfrac{1}{4}m^{-1/2} < \beta^* < 4m^{-1/2}.
\] 

\item We have $x_{p_{2\ell}} < 4(m-1)^{-2\ell-1/2}$ and $x_{c_j} < 8(m-1)^{-2\ell-3/2}$ for all $j=1,2,\ldots ,k$.
\end{enumerate}
\end{lemma}

\begin{proof}[{\bf Proof of (i)}]
We denote the ratio $\rho_i := \frac{x_{p_i}}{x_{p_{i-1}}}$ for each $i$.
The goal is to establish $\rho _i < \frac{1}{\lambda-1}$
for all $i = 1,\dots,2\ell$.
We shall proceed by reverse induction on $i$, starting from $i=2\ell$ to $i=1$.

Recall that $P=\{p_1,\ldots ,p_{2\ell}\}$ and $C=\{c_1,\dots,c_k\}$.
By Definition \ref{defn-construction}, the vertices $c_2,\dots,c_k$ in the clique $C$ are symmetric, so they have the same eigenvector coordinate. We may assume that 
$x_c := x_{c_2}=\dots=x_{c_k}$. 
The eigenvector equation at $c_1$ reads as 
\begin{equation}\label{e-9}
\lambda x_{c_1} = \sum_{u\in N(c_1)}x_u= x_{p_{2\ell}} + (k-1)x_c.
\end{equation}
For any other clique vertex, say $c_j$ with $j\geq 2$, we have
\begin{equation*}
\lambda x_c = \sum_{u\in N(c_j)}x_u = x_{c_1} + (k-2)x_c, 
\end{equation*}
which yields $x_c = \frac{x_{c_1}}{\lambda - k + 2}$.
Substituting $x_c$ into (\ref{e-9}) gives
\[
\lambda x_{c_1} = x_{p_{2\ell}} + \frac{k-1}{\lambda - k + 2}\,x_{c_1},
\]
and rearranging yields 
\begin{equation}\label{e-12}
x_{c_1} = \gamma\, x_{p_{2\ell}}, \quad \text{where} \quad
\gamma := \frac{\lambda - k + 2}{\lambda^2 - (k-2)\lambda - (k-1)}.
\end{equation}  
Since $\lambda^2 - (k-2)\lambda - (k-1)=(\lambda+1)(\lambda-k+1)$ and $\lambda >m \ge 2k $, the denominator is positive, and one can obtain $0 < \gamma < \frac{1}{\lambda-1}$ by a simple calculation.

The eigenvector equation at $p_{2\ell}$ is
\[
\lambda x_{p_{2\ell}} = x_{p_{2\ell-1}} + x_{c_1}.
\]
Using $x_{c_1} = \gamma x_{p_{2\ell}}$ from (\ref{e-12})
and setting
$\rho_{2\ell} := \frac{x_{p_{2\ell}}}{x_{p_{2\ell-1}}}$,
we obtain $(\lambda - \gamma) x_{p_{2\ell}} = x_{p_{2\ell-1}}$ and 
\begin{equation}\label{e-14}
\rho_{2\ell} = \frac{1}{\lambda - \gamma}
< \frac{1}{\lambda - 1},
\end{equation}
where the strict inequality follows directly from $\gamma < \frac{1}{\lambda-1}< \frac{1}{2}$.

For any vertex $p_i$ with $1 \le i \le 2\ell-1$, the eigenvector equation is
$\lambda x_{p_i} = x_{p_{i-1}} + x_{p_{i+1}}$. 
Dividing by $x_{p_i}$ both sides, we obtain
\[
\lambda = \frac{x_{p_{i-1}}}{x_{p_i}} + \frac{x_{p_{i+1}}}{x_{p_i}} = \frac{1}{\rho_i} + \rho_{i+1}.
\]
Rearranging yields the recurrence
\begin{equation}\label{e-15}
\rho_i = \frac{1}{\lambda - \rho_{i+1}}.
\end{equation}
We now prove that
$\rho_i < \frac{1}{\lambda-1}$ by the reverse induction on $i = 2\ell, 2\ell-1, \dots, 1$. 
The base case $i = 2\ell$ is exactly (\ref{e-14}).
Assume that 
the inequality holds for $i+1$, i.e.,
$\rho_{i+1} < \frac{1}{\lambda-1}$.
From (\ref{e-15}), we get 
\[
\rho_i < \frac{1}{\lambda - \frac{1}{\lambda-1}}
= \frac{\lambda - 1}{\lambda^2 - \lambda - 1}.
\]
To complete the induction step, it suffices to check that $\frac{\lambda - 1}{\lambda^2 - \lambda - 1} < \frac{1}{\lambda - 1}$. 
This is trivial since $\lambda > m \ge 6$.
Thus, $\rho_i < \frac{1}{\lambda-1}$.
Consequently, we conclude that for every $i = 1, 2, \dots, 2\ell$, 
\[
x_{p_i} < \frac{x_{p_{i-1}}}{\lambda - 1}.
\]

\noindent\textbf{Proof of (ii).}
Now consider the clique $C = \{c_1, \dots, c_k\}$.
We sum the eigenvector equations over all its vertices.
The vertex $c_1$ is adjacent to $p_{2\ell}$
and the other $k-1$ clique vertices, while each $c_j$ ($j \ge 2$)
is adjacent only to the other $k-1$ clique vertices.
Summing these equations yields
$$\lambda \sum_{i=1}^k x_{c_i} =
x_{p_{2\ell}} + \sum_{i=1}^k \sum_{u \in C, u \neq c_i} x_u.$$
Since the subgraph induced by $C$ is a clique $K_k$,
every vertex coordinate $x_{c_i}$ appears exactly $k-1$ times
in the double sum on the right-hand side.
Hence the equation simplifies to
$$\lambda \sum_{i=1}^k x_{c_i} =
x_{p_{2\ell}} + (k-1) \sum_{i=1}^k x_{c_i}.$$
Rearranging, we obtain the sum of the Perron weights
$$\sum_{i=1}^k x_{c_i} = \frac{x_{p_{2\ell}}}{\lambda - k + 1}.$$
Since every $x_{c_j} > 0$, it follows that
each coordinate $x_{c_j}$ satisfies
$$x_{c_j} < \frac{x_{p_{2\ell}}}{\lambda - k + 1} \quad \text{for all } j = 1, \dots, k.$$

\noindent\textbf{Proof of (iii).} 
For any $v\in S_A$, the neighbors of $v$ are the 
vertices of $Y_A$, so
\begin{equation}\label{eq:SA}
\lambda\,\alpha = (m-1)\,\beta + \beta^*.
\end{equation}
For any $v\in Y_A\setminus\{y^*\}$,
the neighbors are precisely the $m$ vertices of $S_A$, giving
\begin{equation}\label{eq:YA}
\lambda\,\beta = m\,\alpha ~~\text{and}~~\beta = \frac{m\,\alpha}{\lambda}.
\end{equation}
Invoking the fact $m<\lambda < m+1$, 
we get from (\ref{eq:YA}) that 
$$\beta < \alpha ~~\text{and}~~ \beta = \frac{m}{\lambda}\,\alpha
> \frac{m}{m+1}\alpha .$$ 
For the vertex $y^* \in Y_A$, its neighbors consist of all vertices of $S_A$ and the vertex $p_1$,
hence 
\begin{equation}\label{eq:ystar}
\lambda\,\beta^* = m\,\alpha + x_{p_1} = \lambda \, \beta + x_{p_1} ~~\text{and}~~\beta^* = \beta + \frac{x_{p_1}}{\lambda}.
\end{equation}  
Substituting $\beta$ and $\beta^*$ into (\ref{eq:SA})  gives 
$\lambda\,\alpha
= m \beta  + \frac{x_{p_1}}{\lambda}
= m^2\frac{\alpha}{\lambda} + \frac{x_{p_1}}{\lambda}$. Then 
\begin{equation}\label{eq:alpha-xp1}
(\lambda^2 - m^2)\,\alpha = x_{p_1}.
\end{equation}
Using (\ref{eq:YA}), (\ref{eq:ystar}), (\ref{eq:alpha-xp1}) and 
$m <\lambda < m+1$ again, we have
\[
\beta^* 
= \beta + \frac{x_{p_1}}{\lambda}
= \frac{m  + (\lambda^2 - m^2)}{\lambda}\alpha < 
 \frac{3m+1}{m} \alpha \le 4 \alpha. \] 
From the above discussion, we conclude that 
$\frac{3}{4} \alpha < \beta < \alpha$ and 
$\beta < \beta^* < 4\alpha$. 
To prove $\alpha ,\beta ,\beta^* =\Theta(m^{-1/2})$, it suffices to show  $\alpha = \Theta (m^{-1/2})$. 
Clearly,  
we have 
$m \alpha^2 < 1$ and   
$\alpha < m^{-1/2}$.

Since $\bm{x}$ is a unit vector, i.e.,  $\sum_{v\in V} x_v^2 = 1$, 
we have 
\begin{equation}\label{eq:norm}
m\,\alpha^2 + (m-1)\,\beta^2 + (\beta^*)^2 + R = 1,
\end{equation}
where $R := \sum_{i=1}^{2\ell} x_{p_i}^2
+ \sum_{j=1}^{k} x_{c_j}^2 \ge 0$.

From part~(i), iterating the ratio estimate gives
$x_{p_i}< \frac{x_{p_{i-1}}}{\lambda -1}< \cdots < \frac{x_{p_0}}{(\lambda - 1)^i}$
for all $i \ge 1$.
Since $x_{p_0} \le 1$ 
and $\lambda > m  \ge  6$, the contribution of the path to $R$ satisfies 
\begin{equation} \label{eq-path}
\sum_{i=1}^{2\ell} x_{p_i}^2
< \sum_{i=1}^{\infty} \frac{1}{(\lambda - 1)^{2i}} = \frac{1}{(\lambda -1)^2-1} 
\le \frac{1}{(m-1)^2-1}< \frac{2}{m^2}.
\end{equation} 
By part~(ii), each $x_{c_j} < \frac{x_{p_{2\ell}}}{\lambda-k+1}<\frac{1}{m-k+1}$. 
Then 
\[ \sum_{j=1}^k x_{c_j}^2 < \frac{k}{(m-k+1)^2} \le \frac{4k}{m^2},\]
where the last inequality holds since $m\ge 2k$. Therefore
\begin{equation}\label{eq:R-bound}
R < \frac{2}{m^2} + \frac{4k}{m^2} \le \frac{5k}{m^2}.
\end{equation}
Note that $\beta<\alpha$ and $\beta^*< 4\alpha$. By (\ref{eq:norm}), we have 
\[
1 < m\alpha^2+(m-1)\alpha^2+16\alpha^2+R=(2m+15)\alpha^2+R,
\]
which implies 
\[ \alpha^2>\frac{1-(5k)/m^2}{2m+15}\geq 
\frac{1}{9m}.\]
Thus, we obtain that  
$\frac{1}{3}m^{-1/2} < \alpha < m^{-1/2}$.  
Combining with $\frac{3}{4} \alpha <\beta < \alpha$, 
we get $\frac{1}{4}m^{-1/2} < \beta < m^{-1/2}$. 
Moreover, using $\beta < \beta^* < 4\alpha$ yields 
$\frac{1}{4} m^{-1/2} < \beta^* < 4m^{-1/2}$, as desired.

\medskip
\noindent\textbf{Proof of (iv).}
From part (i), we have $x_{p_i}  < \frac{x_{p_{i-1}}}{\lambda - 1}$
for $i=1,\dots,2\ell$. 
Iterating this inequality yields
\[
x_{p_{2\ell}} < \frac{x_{p_{2\ell-1}}}{\lambda-1}
< \frac{x_{p_{2\ell-2}}}{(\lambda-1)^2} < \cdots
< \frac{x_{p_0}}{(\lambda-1)^{2\ell}}.
\]
Recall that $p_0=y^*$ and $x_{y^*} = \beta^* <  4m^{-1/2}$.
Since $\lambda > m$, we have 
$x_{p_{2\ell}}<  4 (m-1)^{-2\ell-1/2}$. 
By part (ii), each coordinate of the clique satisfies
$x_{c_j}<\frac{x_{p_{2\ell}}}{\lambda-k+1} < \frac{4(m-1)^{-2\ell-1/2}}{m-k+1} < 8 (m-1)^{-2\ell - 3/2}$. 
\end{proof}

\subsection{Warm up for a weak bound}

As a warm-up and to illustrate our method, we  next give a quick proof of a weak bound, saying that $\frac{1}{2} - \sigma (S) < \frac{6}{m}$. 
This follows directly from the estimates in Lemma~\ref{lem-2}.

\begin{prop} \label{prop-weak}
    Under (\ref{eq-defn-S}) and Definition \ref{defn-construction}, 
    we have $\frac{1}{2} -  \sigma(S) <  \frac{6}{m}$. 
\end{prop}

\begin{proof}
Recall that $\|\bm{y}\|^2= \sigma (S)$ and $\|{\bm z}\|^2 = 1 - \sigma (S)$,
where
$$\sigma (S)= \sum_{v\in S} x_v^2
= \sum_{v\in S_A} x_v^2 + \sum_{v\in S_{\mathrm{path}}} x_v^2.$$ 
Using (\ref{eq-path}), we obtain
$\sum_{v\in S_{\mathrm{path}}} x_v^2
< \sum_{i=1}^{2\ell} x_{p_i}^2 < {2}/{m^2}$, which yields 
\begin{equation} \label{eq-sigma}
m\alpha^2 < \sigma (S) < m\alpha^2 + \frac{2}{m^2}. 
\end{equation}
Recall in (\ref{eq:R-bound}) that $R < {(5k)}/{m^2}$.
Then 
\[ 1- \sigma (S) = (m-1)\beta^2 + (\beta^*)^2
+ \sum_{v\in \bar{S}_{\mathrm{path}}} x_v^2
+ \sum_{v\in C} x_v^2 < (m-1)\beta^2 + (\beta^*)^2 
+ \frac{5k}{m^2}.
\]
We next show that $\sigma (S)$ and $1-\sigma (S)$ are
almost equal. From the above discussion, our goal is to estimate the difference between $m\alpha^2$ and $(m-1)\beta^2 + (\beta^*)^2$. 
From (\ref{eq:YA}) and (\ref{eq:ystar}), we have 
\begin{equation*}
(m-1)\beta^2 + (\beta^*)^2
= (m-1)\frac{m^2\alpha^2}{\lambda^2}
  + \left(\frac{m\alpha}{\lambda}
    + \frac{x_{p_1}}{\lambda}\right)^{2} 
    = \frac{m^3\alpha^2}{\lambda^2}
   + \frac{2m\alpha\, x_{p_1}}{\lambda^2}
   + \frac{x_{p_1}^2}{\lambda^2}.
\end{equation*}
It follows that 
\begin{align}
(m-1)\beta^2 + (\beta^*)^2 - m\alpha^2
= m\alpha^2\!\left(\frac{m^2}{\lambda^2} - 1\right)
   + \frac{2m\alpha\, x_{p_1}}{\lambda^2}
   + \frac{x_{p_1}^2}{\lambda^2} .\label{eq:diff-main}
\end{align}
Recall in (\ref{eq:alpha-xp1}) that 
$x_{p_1} = (\lambda^2 - m^2)\alpha$.
Substituting this into (\ref{eq:diff-main}) yields
\begin{align*}
(m-1)\beta^2 + (\beta^*)^2 - m\alpha^2
&= -\,\frac{m\alpha^2(\lambda^2-m^2)}{\lambda^2}
   + \frac{2m\alpha^2(\lambda^2-m^2)}{\lambda^2}
   + \frac{\alpha^2(\lambda^2-m^2)^2}{\lambda^2} \\ 
&= \frac{\alpha^2(\lambda^2-m^2)(m + \lambda^2-m^2)}{\lambda^2}.
\end{align*}
Since $m< \lambda < m+1$, we have $0 < \lambda^2 - m^2 < 2m+1$.
Together with $\alpha^2 < \frac{1}{m}$, we obtain
\begin{equation*}
\left|(m-1)\beta^2 + (\beta^*)^2 - m\alpha^2\right|
< \frac{\frac{1}{m}\cdot(2m+1)\cdot(3m+1)}{m^2} <  \frac{7}{m}.
\end{equation*} 
Consequently, we get 
\begin{equation} \label{eq-1-sigma} 
1-\sigma(S) < (m-1)\beta^2 + (\beta^*)^2 +
\frac{5k}{m^2} < m\alpha^2 + 
\frac{10}{m}.
\end{equation} 
Adding (\ref{eq-sigma}) and (\ref{eq-1-sigma}) yields
$1 < 2m\alpha^2 + \frac{12}{m}$, 
which implies 
$  \frac12 - \frac{6}{m} < m\alpha^2 < \sigma (S)$, as needed.  
\end{proof}

\noindent 
{\bf Remark.} 
In Definition \ref{defn-construction}, setting $m = \Theta(n)$ and combining with Proposition \ref{prop-weak}, we obtain $\frac{1}{2} - \sigma(S) < O(1/n)$, which is far smaller than the right-hand side of (\ref{eq-Gre}). Hence, even this weak bound of Proposition \ref{prop-weak} suffices to refute Gregory's bound in  Conjecture \ref{conj-Gre}.

\subsection{Proof of Theorem \ref{thm-main-2}}

In Proposition \ref{prop-weak}, we obtained a weak bound $\frac{1}{2} - \sigma (S) < \frac{6}{m}$. To improve this bound and establish an exponentially small gap, a key ingredient is to utilize the exact identities from Theorems \ref{thm-identity} or \ref{thm:sum-uv}. 
Here, we prefer to use the simple version in Theorem \ref{thm:sum-uv}, which states 
\begin{equation} \label{eq-sec4-2-identity}
\frac12 - \sigma(S) =\frac{1}{\lambda} \sum_{uv\in E(\bar{S}) } x_ux_v .
\end{equation}
We now use the estimates provided by Lemma~\ref{lem-2}.

\begin{claim} \label{cl:shang}
    We have $\sum_{uv\in E(\bar{S})} x_ux_v \le  30k\cdot (m-1)^{-4\ell -2}$.
\end{claim} 

\begin{proof}[Proof of claim]
Recall that $\bar{S}= Y_A \cup \bar{S}_{\mathrm{path}} \cup C$, where  
$Y_A$ is an independent set of $K_{m,m}$,
and $\bar{S}_{\mathrm{path}} = \{p_2, p_4, \dots, p_{2\ell}\}$
consists of even-indexed vertices,
which are pairwise non-adjacent, and $C$ forms a clique $K_k$. Therefore, 
the edges of $G[\bar{S}]$ are of the following two types: 
\begin{itemize}
\item the $\binom{k}{2}$ edges inside the clique $C$: By Lemma \ref{lem-2} (iv), and $m\ge 2k$, we obtain  
$$ \sum_{uv \in E(C)} x_u x_v \le  \binom{k}{2}\cdot 64(m-1)^{-4\ell-3} < 16 k \cdot (m-1)^{-4\ell -2}.$$
\item the bridge edge $p_{2\ell}c_1$: The contribution of this edge is given as 
$$ x_{p_{2\ell}} x_{c_1} \le 
 4(m-1)^{-2\ell -1/2} \cdot 8(m-1)^{-2\ell -3/2}\le 32 (m-1)^{-4\ell-2}.$$
\end{itemize}
So we obtain 
$\sum_{uv\in E(\bar{S})} x_ux_v
= \sum_{uv\in E(C)} x_u x_v + x_{p_{2\ell}}\,x_{c_1} \le 30k \cdot (m-1)^{-4\ell-2}$, as needed.  
\end{proof}

Applying the identity (\ref{eq-sec4-2-identity}), 
and combining with Claim \ref{cl:shang}, we get 
\begin{equation*}
\frac12 - \sigma(S) \le \frac{30k}{\lambda \,(m-1)^{4\ell +2}} <  
\frac{30k}{(m-1)^{4\ell+3}}, 
\end{equation*}
where the last inequality holds since $\lambda >m$. 
This completes the proof of Theorem \ref{thm-main-2}.

\subsection{Proof of Theorem \ref{thm-main}}

In what follows, we show that Theorem \ref{thm-main} is a direct consequence of Theorem \ref{thm-main-2}.  Indeed, we assume that $k$ is fixed and $n$ is sufficiently large. 
In the case that $n-k$ is even, 
we can choose $m :=\lfloor {(n-k)}/{4}\rfloor$ and $\ell := \lceil {(n-k)}/{4} \rceil$. 
Then $2m + 2\ell + k = n $. So Theorem \ref{thm-main-2} gives 
$$ \frac{1}{2} - \sum_{v\in S} x_v^2 
= O_k(m^{-4\ell -3}) = e^{-\Omega_k(n \log n)}. $$  
To obtain the bound for all sufficiently large $n$, 
we use the same construction of Definition \ref{defn-construction} with a connecting path of order $2\ell + \varepsilon$, where $\varepsilon \in \{0, 1\}$ is chosen so that $n - k - \varepsilon$ is even. Define $S := S_A \cup \{p_i : i \text{ is odd}\}$. The estimates in Lemma \ref{lem-2} and Claim \ref{cl:shang} are unchanged up to absolute constants, with $2\ell$ replaced by $2\ell + \varepsilon$. Taking $m = \lfloor (n - k - \varepsilon)/4 \rfloor$ and $\ell = (n - k - \varepsilon)/2 - m$ gives a graph on exactly $n$ vertices and yields 
$\frac{1}{2} - \sum_{v \in S} x_v^2 \leq e^{-\Omega_k(n \log n)}$. This completes the proof.

\section{Tightness of Theorem \ref{thm-main} and lower bounds}

\label{sec:tight}

In Theorem \ref{thm-identity}, 
we obtained that for any independent set $S$, 
\begin{equation} \label{eq-sec5-identity}
\frac{1}{2} - \sigma (S) =  \frac{q}{4\lambda - 2q},
\end{equation}
where $q = \frac{{\bm{z}^{\mathsf{T}} A_{\bar{S}}\bm{z}}}{\|\bm{z}\|^2}$ is the Rayleigh quotient of the Perron vector on $\bar{S}$. 
Thus, bounding the gap $\frac{1}{2} - \sigma (S)$ from below reduces to bounding quadratic form $q$ from below. As an application, we now prove that the bound in Theorem~\ref{thm-main} {\it cannot} be improved beyond a factor in the exponent.

\begin{thm} \label{thm-lower}
If $G$ is a non-bipartite graph of order $n$,  
 then for any independent set $S$ of $G$, 
\[ \frac{1}{2} - \sum_{v\in S} x_v^2 \ge e^{-O(n\log n)}. \]  
 \end{thm}

 \begin{proof}
Since $G$ is non-bipartite, we see that $\bar{S}$ contains at least one edge $uv \in E(G)$. This gives 
${\bm{z}^{\mathsf{T}} A_{\bar{S}}\bm{z}}=2\sum_{ij\in E(\bar{S})}x_{i} x_{j}\geq 2x_{u}x_{v}$. 
From the eigenvalue equation $A \bm{x} = \lambda \bm{x}$, for the edge $uv$, we have 
$x_v = \frac{1}{\lambda} (x_u+\sum_{w\sim v,w\neq u}x_w)\ge \frac{1}{\lambda}x_u$. It follows that  
$x_v \ge \lambda^{-1} x_u$.

Assume that all coordinates of the Perron vector are sorted in the order $x_{v_1}\ge \cdots \ge x_{v_n}$.  
Let $d(v_1,v_n)$ be the distance between $v_1$ and $v_n$. 
By inducting a shortest path starting from $v_1$ to $v_n$, the above discussion gives   
$x_{v_n} \ge \lambda^{-d(v_1,v_n)} 
x_{v_1}$.  
Invoking the fact $\sum_{i=1}^n x_{v_i}^2=1$, we obtain 
$x_{v_1}\ge {1}/{\sqrt{n}}$. We write $D$ for the diameter of $G$. 
Thus, we conclude that $x_{v_n} \ge \lambda^{-D}/ \sqrt{n} $. Then 
$$
q = \frac{{\bm{z}^{\mathsf{T}} A_{\bar{S}}\bm{z}}}{{\|\bm{z}\|^2}}\geq {\bm{z}^{\mathsf{T}} A_{\bar{S}}\bm{z}} \geq 2x_u x_v \geq \frac{2}{\lambda^{2D}n}.
$$
Since $D<n$ and $\lambda\leq \Delta<n$, we have  $\frac{2}{\lambda^{2D}n}>\frac{2}{n^{2n+1}}$. Combining with (\ref{eq-sec4-2-identity}), 
we obtain 
\[ \frac{1}{2} - \sigma (S) > \frac{q}{4n}\ge \frac{1}{2n^{2n+2}} = e^{-O(n\log n)}. \qedhere \] 
\end{proof}

It is worth noting that even in the non-bipartite case, the above lower bound can be quite weak. 
As an illustration, we consider the complete $k$-partite graph $K_{n/k,\ldots,n/k}$, where $k$ divides $n$. Then $D=2$ and $\lambda = \frac{k-1}{k}n$. By symmetry, we obtain $x_v=1/\sqrt{n}$. 
So a maximum independent set gives $\sigma(S)=\frac{n}{k}\cdot \frac{1}{n}=\frac{1}{k}$ and 
$\frac{1}{2} - \sigma (S)=\frac{1}{2}-\frac{1}{k}$. 
However, in the proof of Theorem \ref{thm-lower}, 
we see that $q\ge {2}{\lambda^{-2D}/n}$. Then (\ref{eq-sec4-2-identity}) gives $\frac{1}{2} - \sigma (S)= \Omega(1/n^{6})$, which is a polynomial bound.

\smallskip 
The construction in Theorem \ref{thm-main} exhibits exponential decay due to its diameter being $\Theta(n)$. For graphs with bounded diameter, such as complete $k$-partite graphs, the gap stays polynomial in terms of $n$. 
This difference suggests that the size of the gap may be determined by the diameter, rather than the chromatic number. Inspired by this observation, a more appropriate framework is possibly to fix the diameter $D$, and ask how the gap $\frac{1}{2}- \sigma (S)$ depends on $D$ and $n$ jointly. 

\smallskip 
In the sequel, we provide some lower bounds for $\frac{1}{2} - \sigma (S)$, where $S$ is an arbitrary independent set of $G$.  
Applying Theorem \ref{thm:sum-uv}, we obtain 
a local weighted lower bound.

\begin{thm} \label{thm:local}
Let $G$ be a connected graph of order $n$, let $\bm{x}$ be its unit Perron vector, let $S$ be an independent set. Let $\lambda=\lambda(G)$, choose a vertex $v^\ast\in V(G)$ such that
$x_{v^\ast}:=\max\{x_v: v\in V(G)\}$,
and define
$\rho (u):=\operatorname{dist}_G(u,v^*)$ for each $u\in V(G)$. Then
\[
\frac{1}{2} - \sum_{v\in S}x_v^2
\ge\frac{1}{n}
\sum_{uv\in E(\bar{S})}\lambda^{-\rho (u)- \rho (v) -1}.
\]
\end{thm}

\begin{proof}
For a vertex $u\in V(G)$, let $v^\ast := w_0 \to w_1 \to \cdots \to w_\ell := u$ (where $\ell =\rho (u)$) 
be a shortest path starting from $v_\ast$ to $u$. For each $1\le i\le \ell$, the eigenvalue equation at $w_i$ gives $\lambda x_{w_i}=\sum_{z\sim w_i}x_z\ge x_{w_{i-1}}$, 
hence $x_{w_i}\geq x_{w_{i-1}} \lambda^{-1}$.
Iterating along the path yields $x_u\geq x_{v^*}\,\lambda^{- \rho (u)}$. 
Note that $x_{v^*}^2 \ge 1/n$ since $\bm{x}$ is unit. 
Therefore, for every edge $uv$ in $G[\bar{S}]$, 
\[
x_u x_v\geq x_{v^*}^2\,\lambda^{- \rho(u)-\rho (v)}
\geq \frac{1}{n} \lambda^{- \rho(u)- \rho(v)}.
\]
Combining with Theorem~\ref{thm:sum-uv}, we obtain
\[
\frac{1}{2} - \sum_{v\in S}x_v^2=\frac1{\lambda}\sum_{uv\in E(\bar{S})}x_u x_v
\geq \frac{1}{n}\sum_{uv\in E(\bar{S})}\lambda^{-\rho(u)-\rho(v) -1},
\]
which gives the desired inequality.
\end{proof}

Next, we give an corollary of Theorem \ref{thm:local}. We present a lower bound on $\frac{1}{2} - \sum_{v\in S} x_v^2$ in terms of the diameter and the chromatic number of $G$. 

\begin{cor} \label{cor:radius}
Under the hypotheses of Theorem~\ref{thm:local}, 
let  $D$ be the diameter and $\chi(G)=k$ be the chromatic number of $G$, where $k\ge 3$. Then
\[
\frac{1}{2} - \sum_{v\in S}x_v^2 
\ge \frac{\binom{k-1}{2}}{n\,\lambda^{2 D+1}}.
\]
\end{cor}

\begin{proof}
For every edge $uv\in E(\bar{S})$, we have  $\rho (u), \rho(v)\le D$. 
Now assume $\chi(G)=k\ge 3$. Since $S$ is independent, we have $\chi(G[\bar{S}])\ge k-1$. 
A graph $H$ is {\it \it color-critical} 
if $\chi (H')<\chi (H)$ for every proper subgraph $H'$ of $H$. By taking a color-critical subgraph of $G[\bar{S}]$, 
we may assume $\delta (H)\ge \chi (H)-1$. Otherwise, if $\delta (H)\le \chi (H) -2$, then 
let $v\in V(H)$ be a vertex with minimum degree, since $H\setminus \{v\}$ can be colored by at most $\chi (H)-1$ colors, and the neighbors of $v$ use at most $\chi (H)-2$ colors, so $H$ can be colored by $\chi (H)-1$ colors, a contradiction.  

From the above discussion, we have $\delta (H)\ge \chi (H) -1 \ge k-2$, so
\[
2e(H)\ge (k-2)|V(H)|\ge (k-2)(k-1).
\]
Hence
\[
e(G[\bar{S}])\ge e(H)\ge \binom{k-1}{2}.
\]
Applying Theorem \ref{thm:local} gives
$\frac{1}{2} - \sum_{v\in S}x_v^2\ge 
{\binom{k-1}{2}}/ {\big(n\,\lambda^{2 D+1} \big)}$, as needed. 
\end{proof}

\begin{cor} \label{cor:coarse}
Let $G$ be a connected graph of order $n$ with chromatic number $k\ge 3$, let $\bm{x}$ be its unit Perron vector, and let $S$ be an independent set. Then
\[
\frac{1}{2} - \sum_{v\in S}x_v^2 
> \frac{k-2}{n^{2n-2k+5}}.
\]
\end{cor}

\begin{proof}
Let $\bar{S}:=V(G)\setminus S $ and  $v^\ast\in V(G)$ be a vertex defined as in Theorem~\ref{thm:local}. Let $H$ be a color-minimal subgraph of $G[\bar{S}]$ with $\chi(H)\ge k-1$. Such a subgraph exists because $\chi(G[\bar{S}])\ge k-1$. 

Choose a vertex $u\in V(H)$ such that  the distance $d( v^*,u)$ is minimum. Denote $d:=d(v^*, u)$. From the proof of Corollary \ref{cor:radius}, we know that $d_H(u)\ge k-2$. 
For each $w\in N_H(u)$, we get  $d(v^*,w)\le d(v^*,u) +d(u,w) \le d+1$. Applying Theorem~\ref{thm:local} yields
\[
\frac{1}{2} - \sum_{v\in S}x_v^2\ge 
\frac{1}{n} \sum_{w\in N_H(u)} \frac{1}{\lambda^{\rho(u) + \rho (w) +1}} 
\ge \frac{k-2}{n\,\lambda^{2d+2}}.
\]
It remains to bound $d$. By the minimality of $d(v^*,u)$, a shortest path from $v^*$ to $u$ meets $H$ only at its last vertex $u$. Thus, this path contributes $d$ vertices outside $H$, so
$n\ge d+|V(H)|$. 
Since $\chi(H)\ge k-1$, we have $|V(H)|\ge k-1$, and therefore $d\le n-k+1$. 
Consequently, we get 
\[
\frac{1}{2} - \sum_{v\in S}x_v^2\ge \frac{k-2}{n\,\lambda^{2n-2k+4}}.
\]
Finally, invoking the fact $\lambda(G)\le \Delta(G) < n$, we obtain the desired inequality.
\end{proof}

\section{Concluding remarks}

\label{sec:concluding}

In this paper, 
we first showed that cycles $C_n$ (with odd integers $n\ge 7$) serve as  counterexamples to Gregory's conjecture; see Example \ref{example-1}.  This is much simpler than the construction recently provided by Liu and Ning \cite{LN2026}. 
Second, we established two identities for  
$\frac{1}{2} - \sum_{v\in S}x_v^2$; see Theorems \ref{thm-identity} and \ref{thm:sum-uv}.  
These identities capture the intrinsic relation between the gap $\frac{1}{2} - \sum_{v\in S}x_v^2$
 and the graph structure of $G[\bar{S}]$. 
Third, we constructed a graph $G$ with chromatic number $k$ and found an independent set $S$ such that $\sum_{v\in S}x_v^2$ can be arbitrarily close to $1/2$, with a small gap decaying as $\exp(-\Omega(n\log n))$; see Theorems \ref{thm-main} and \ref{thm-main-2}. 
Finally, we illustrated that the bound in Theorem \ref{thm-main} is tight up to a constant factor in the exponent; see Theorem \ref{thm-lower}. 
Moreover, we provide some local weighted lower bounds involving the distances of vertices; see Theorem \ref{thm:local}.  

\smallskip 
Our construction essentially consists of two parts. One part is a large complete bipartite graph, which is designed to make $\sum_{v\in S}x_v^2$ close to $1/2$. The other part is a clique, which is used to satisfy the chromatic number constraint; see Definition \ref{defn-construction}. 
As for the key techniques, we employ a long path to attenuate the contributions of the vertices outside the bipartite graph, and then apply the above identity to bound the gap 
$\frac{1}{2} - \sum_{v\in S}x_v^2$ accurately. 
Our construction not only completely disproves the polynomial lower bound conjectured by Liu and Ning, but also demonstrates that $\sum_{v\in S}x_v^2$ cannot be separated from $1/2$ by any function of polynomial order in $n$ and $k$.

\smallskip
Let $C_{n,k}$ be the minimum value of the gap  $\frac{1}{2} - \sum_{v\in S}x_v^2$, where $G$ is an $n$-vertex connected graph with chromatic number $k\ge 3$ and unit Perron vector $\bm{x}$, and $S$ is an independent set of $G$. 
Combining Theorems \ref{thm-main} and \ref{thm-lower}, we see that $\log C_{n,k} = - \Theta_k (n \log n)$. Therefore, it would be interesting to determine an exact constant $\delta_k$ such that $\log C_{n,k}= - \big(\delta_k +o(1) \big)\, n\log n$.

\begin{thebibliography}{99}

\bibitem{BFP2008}
A. Bhattacharya, S. Friedland, U.N. Peled, 
On the first eigenvalue of bipartite graphs, 
Electron. J. Combin. 15 (2008), Paper No. R144. 

\bibitem{Bil2006}
Y. Bilu, 
Tales of Hoffman: Three extensions of Hoffman's bound on the graph chromatic number, 
J. Combin. Theory Ser. B 96 (2006) 608--613. 

\bibitem{Cio2010} 
S.M. Cioab\u{a}, A necessary and sufficient eigenvector condition for a connected graph to be bipartite,
Electron. J. Linear Algebra 20 (2010) 351--353.

\bibitem{Cio2021} 
S.M. Cioab\u{a}, The principal eigenvector of a connected graph, in: Spectral Graph Theory Online
Conference, (2021), \url{http://spectralgraphtheory.org/sgt-online-program}.

\bibitem{CDT2024}
S.M. Cioabă, D.N. Desai, M. Tait, 
The spectral even cycle problem, 
Combin. Theory 4 (1) (2024), Paper No. 10. 

\bibitem{CH2026}
P. Csikvári, V. Harangi, 
The least balanced graphs and trees, 
J. Combin. Theory Ser. B 179 (2026) 147--204.

\bibitem{GN2008}
C.D. Godsil and M.W. Newman, Eigenvalue bounds for independent sets, J. Combin. Theory Ser. B 98 (2008) 721--734. 

\bibitem{Hae2021}
W.H. Haemers, Hoffman's ratio bound, Linear Algebra Appl. 617 (2021) 215--219.

\bibitem{Hua2019}
H. Huang, Induced subgraphs of hypercubes and a proof of the sensitivity conjecture, 
Ann. of Math. (2) 190 (3) (2019) 949–955.

\bibitem{JTYZZ2021}
Z. Jiang, J. Tidor, Y. Yao, S. Zhang, Y. Zhao, Equiangular lines with a fixed angle, 
Ann. of Math. (2) 194 (3) (2021) 729–743.

\bibitem{KW2024}
M. Kwan, Y. Wigderson, 
The inertia bound is far from tight, 
Bull. Lond. Math. Soc. 56 (10) (2024) 3196--3208.


\bibitem{LZZ2025}
S. Li, S. Zhao, L. Zou, Spectral extrema of graphs with fixed size: forbidden a fan graph, a friendship graph, or a theta graph, J. Graph Theory 110 (4) (2025) 483--495.

\bibitem{LZS2024}
X. Li, M. Zhai, J. Shu, 
A Brualdi--Hoffman--Turán problem on cycles,
European J. Combin. 120 (2024), Paper No. 103966.

\bibitem{LLZ-spectral-ESS}
Y. Li, H. Liu, S. Zhang, 
An edge-spectral Erd\H{o}s--Stone--Simonovits theorem and its stability, (2025), arXiv:2508.15271.


\bibitem{LLZ-color-critical}
Y. Li, H. Liu, S. Zhang, 
Edge-spectral Tur\'{a}n theorems for color-critical graphs with applications, (2025), arXiv:2511.15431. 

\bibitem{LN2023}
L. Liu, B. Ning, Unsolved problems in spectral
graph theory, Oper. Res. Trans., 27 (2023) 33--60.

\bibitem{LN2026} 
L. Liu, B. Ning, Local properties of the spectral radius and Perron
vector in graphs, 
J. Combin. Theory Ser. B 176 (2026) 241--253.


\bibitem{Tait2019}
M. Tait, 
The Colin de Verdi\'{e}re parameter, excluded minors, and the spectral radius, 
J. Combin. Theory Ser. A 166 (2019) 42--58.

\bibitem{TZE2025}
Q. Tang, S. Zhang, C. Elphick, 
Inertia, independence and expanders, 
Bull. Lond. Math. Soc. 57 (12) (2025) 4076--4095. 

\bibitem{Wil1967} 
H.S. Wilf, 
The eigenvalues of a graph and its chromatic number, 
J. London Math. Soc. 42 (1967) 330--332. 

\bibitem{WE2013} 
P. Wocjan, C. Elphick, 
New spectral bounds on the chromatic number encompassing all eigenvalues 
of the adjacency matrix, 
Electronic J. Combin. 20(3) (2013), \# P39. 


\end{thebibliography}
\end{document}